\documentclass[a4paper,UKenglish,cleveref, autoref, thm-restate, final]{lipics-v2021}
\pdfoutput=1 %uncomment to ensure pdflatex processing (mandatatory e.g. to submit to arXiv)
\hideLIPIcs  %uncomment to remove references to LIPIcs series (logo, DOI, ...), e.g. when preparing a pre-final version to be uploaded to arXiv or another public repository

\graphicspath{{./figures/}}%helpful if your graphic files are in another directory

\input{01_packages.sty}
\input{01_arxiv_commands.sty}
\title{Edge densities of drawings of graphs with one forbidden cell} 

\author{Benedikt Hahn}{Graz University of Technology, Austria}{benedikt.hahn@tugraz.at}{https://orcid.org/0009-0004-2325-6997}{This research was funded in whole or in part by the Austrian Science Fund (FWF) 10.55776/DOC183.}

\author{Torsten Ueckerdt}{Karlsruhe Institute of Technology, Germany}{torsten.ueckerdt@kit.edu}{https://orcid.org/0000-0002-0645-9715}{funded by the Deutsche Forschungsgemeinschaft (DFG, German Research Foundation) – 520708409.}

\author{Birgit Vogtenhuber}{Graz University of Technology, Austria}{birgit.vogtenhuber@tugraz.at}{https://orcid.org/0000-0002-7166-4467}{This research was funded in whole or in part by the Austrian Science Fund (FWF) 10.55776/DOC183.}

\authorrunning{Benedikt Hahn, Torsten Ueckerdt, and Birgit Vogtenhuber} 

\Copyright{Benedikt Hahn, Torsten Ueckerdt, and Birgit Vogtenhuber} 



\nolinenumbers 

\EventEditors{}
\EventNoEds{0}
\EventLongTitle{}
\EventShortTitle{}
\EventAcronym{}
\EventYear{}
\EventDate{}
\EventLocation{}
\EventLogo{}
\SeriesVolume{}
\ArticleNo{}
\begin{document}
\maketitle

\begin{abstract}
    A connected topological drawing of a graph divides the plane into a number of cells.
    The type of a cell $c$ is the cyclic sequence of crossings and vertices along the boundary walk of $c$.
    For example, all triangular cells with three incident crossings and no incident vertex share the same cell type.
    When a non-homotopic drawing of an $n$-vertex multigraph $G$ does not contain any such triangular cell, Ackerman and Tardos [JCTA 2007] proved that $G$ has at most $8n-20$ edges, while Kaufmann, Klemz, Knorr, Reddy, Schröder, and Ueckerdt [GD 2024] showed that this bound is tight.
    
    In this paper, we initiate the in-depth study of $\celltype$\emph{-free} drawings, that is, drawings that do not contain any cell of one fixed cell type \celltype, and investigate the edge density of the corresponding graphs, i.e., the maximum possible number of edges.
    We consider non-homotopic as well as simple drawings, multigraphs as well as simple graphs, and every possible cell type $\celltype$.
    For every combination of drawing style, graph type, and cell type, we give upper and lower bounds on the corresponding edge density.
    With the exception of the cell type with four incident crossings and no incident vertex, we show for every cell type \celltype that the edge density of $n$-vertex (multi)graphs with \celltype-free drawings is either linear in $n$ or superlinear in $n$.
    In most cases, our bounds are tight up to an additive constant.
    We further consider the question which simple graphs admit a simple drawing without some given cell type(s). 
    For the class of cell types that are not incident to any crossing, we give a complete characterization of all simple graphs that admit a simple drawing without any such cell.
                
    Additionally, we improve the current lower bound on the edge density of simple graphs that admit a non-homotopic quasiplanar drawing from $7n-28$ to $7.5n-28$.
\end{abstract}

%%%%%%%%%%%%%%%%%%%%%%%%%%
%%                      %%
%%     INTRODUCTION     %%
%%                      %%
%%%%%%%%%%%%%%%%%%%%%%%%%%
\section{Introduction} 
\label{sec:introduction}

Many fundamental classes of graphs are defined by admitting a drawing without a forbidden configuration or pattern.
This includes planar graphs (forbidding two crossing edges), as well as many beyond-planar graph classes such as $k$-planar graphs (forbidding an edge with $k+1$ crossings), quasiplanar graphs (forbidding three pairwise crossing edges), fan-planar graphs (forbidding an edge crossed by two independent edges\footnote{The actual definition is more complex but irrelevant here.}) and many others.
Planar and beyond-planar graphs are of eminent importance, both from a practical and a theoretical point of view \cite{didimoSurveyBeyondGraphDrawing2020}.
Nevertheless, many fundamental properties are still not sufficiently well understood.
One of the most fundamental questions in this area concerns the \emph{edge~density} of a beyond-planar graph class $\mathcal{G}$, that is, the maximum possible number of edges of any $n$-vertex graph in~$\mathcal{G}$.
There is a long history of edge density results, ranging from $k$-planar graphs~\cite{pachGraphsDrawnFew1997,pachImprovingCrossingLemma2006,ackermanTopologicalGraphsMost2019}, $k$-quasiplanar graphs~\cite{ackerman_quasiplanar,ackermanMaximumNumberEdges2009}, $k$-bend RAC graphs~\cite{didimoDrawingGraphsRight2011,angeliniRACDrawingsGraphs2020}, $(k,l)$-grid free graphs~\cite{pachTopologicalGraphsNo2005}, and more to the recently introduced Density Formula~\cite{kaufmann_density_formula}. 

We study the \emph{cells} of a drawing, that is, the regions of $\mathbb{R}^2$ delimited by the vertices and edges of the drawing. 
The \emph{type} of a cell $c$ is the cyclic order of edge segments, vertices, and crossings along the boundary of $c$.
For example, a \threezero-cell is incident to three edge segments, three crossings, and no vertex.
Distinguishing cells by their type is often useful when reasoning about drawings (see for example \cite{harborthEdgesWithoutCrossings1974}).
The recently developed \emph{Density Formula}~\cite{kaufmann_density_formula}, whose various applications usually come down to counting and relating different cell types in drawings, has further increased the need for a good grasp on the relation between types of cells in a drawing and the structure of the underlying graph.
Interestingly, Ackerman and Tardos~\cite{ackerman_quasiplanar} proved already in 2007, while considering quasiplanar drawings, that non-homotopic $n$-vertex drawings with no \threezero-cell contain at most $8n-20$ edges.

%%%%%%%%%%%%%%%%%%%%%%%%%%
%%   OUR CONTRIBUTION   %%
%%%%%%%%%%%%%%%%%%%%%%%%%%
\subparagraph*{Our contribution.}

In this paper, we consider for each cell type \celltype{} the corresponding beyond-planar graph class $\mathcal{G}_{\celltype}$ consisting of all graphs that admit a drawing without \celltype-cells.

As our main results, we present upper and lower bounds on the edge densities of graphs in~$\mathcal{G}_{\celltype}$ when arbitrary drawings are admitted, and when drawings are required to be simple or non-homotopic\footnote{In a simple drawing, two edges can intersect at most once, while non-homotopic have a similar but slightly weaker restriction. Both notions are formally defined in \cref{sec:prelims}}.
For each of these settings and each cell type (with one exception), we determine whether the asymptotic edge density is linear or superlinear.
Moreover, in most cases, our bounds are tight up to an additive constant.
Our results and some previous related work are summarized in \cref{tab:forbidden_cells_summary}.
To obtain the lower bounds, we present various novel constructions, while our upper bounds are proved via discharging arguments.

Further, we study the question, which cell types can be avoided in simple drawings of (essentially) all graphs.
In particular, we fully characterize all connected simple graphs in $\mathcal{G}_\celltype$ for cell types \celltype{} that are not incident to any crossing,  showing that $\mathcal{G}_\celltype$ contains nearly all simple connected graphs.

Finally, we consider the well-studied class of quasiplanar drawings, via its relation to \threezero-free drawings.
For non-homotopic drawings of multigraphs, it is known that the edge densities of these two classes coincide.
We show that despite this fact, the two classes are different.
We further present a construction of simple graphs with $7.5n-28$ edges that admit a non-homotopic quasiplanar drawing, thus halving the gap between the upper and the previously best lower bound for this class. 

\begin{table}[htb]
    \centering
    \begin{tabular}{c|c|c|c|c}
        \toprule
        Forbidden cell& Graph type & Drawing type & Lower Bound & Upper Bound \\
        \midrule
        Any & Simple & No restriction & 
			\customstack{$\bm{{n \choose 2}}$}
			{\cref{con:arbitrary_quadratic_density}}
            & $\binom{n}{2}$\\
        \midrule
        \multirow{3}{*}[-1.6em]{\threezero-cell} & Multi & Non-homotopic &  
            \customstack{$8n-20$}{\cite{kaufmann_density_formula}} &
            \customstack{$8n-20$}{\cite{ackerman_quasiplanar}}\\
        \cmidrule{2-5}
        & Simple & Non-homotopic &  
            \customstack{$\bm{8n-28}$}{\cref{con:simple_nh_threezero_free}}
            &
            \customstack{$8n-20$}{\cite{ackerman_quasiplanar}}\\
        \cmidrule{2-5}
        & - & Simple &  
            \customstack{$7n-30$}{Cons. \ref{con:no_threezero_simple_dense} based on \cite{ackerman_quasiplanar}} &
            \customstack{$7n-20$}{\cite{ackerman_quasiplanar}
            }\\
        \midrule
        \multirow{3}{*}[-1.6em]{\fourzero-cell} & Multi & Non-homotopic &  
            \customstack{$\bm{9n-18}$}{\cref{con:dense_multi_fourzero_free}} &
            $\infty$ \\
        \cmidrule{2-5}
        & Simple & Non-homotopic &  
            \customstack{$\bm{6n-12}$}{\cref{con:dense_fourzero_free}} &
            \multirow{2}{*}[-0.8em]{$\binom{n}{2}$} \\
        \cmidrule{2-4}
        & - & Simple &  
            \customstack{$\bm{6n-12}$}{\cref{con:dense_fourzero_free}} &
             \\
        \midrule
        \multirow{2}{*}[-0.8em]{\fourone-cell} & Simple & Non-homotopic & 
            \customstack{$\bm{\Omega(n^2)}$}{\cref{con:fourone-non-hom}} &
            \multirow{2}{*}[-0.8em]{$\binom{n}{2}$} \\
        \cmidrule{2-4}
        & - & Simple &
            \customstack{$\bm{\Omega(n^{3/2})}$}{\cref{con:fourone_simple_plane}} &
            \\
        \midrule
        \multirow{2}{*}[-0.8em]{\fivezero-cell}  & Multi & Non-homotopic & 
            \customstack{$\bm{6n-12}$}{\cref{con:dense_fivezero_free}} & 
            \customstack{$\bm{6n-12}$}{\cref{thm:nh_fivezero_free}}\\
        \cmidrule{2-5}
        & - & Simple & 
            \customstack{$\bm{6n-12}$}{\cref{con:dense_fivezero_free}} & 
            \customstack{$\bm{6n-12}$}{\cref{thm:nh_fivezero_free}}\\
        \midrule
        Others & - & Simple & 
            \customstack{$\bm{{n \choose 2}}$}{\cref{con:most_cell_types_avoidable,con:no-fivetwo-cell-complete-graphs}} &
            $\binom{n}{2}$\\
        \bottomrule
    \end{tabular}
    \medskip
    \caption{
    Results on edge densities of $n$-vertex graphs admitting a drawing without one fixed type of cell. Upper bounds hold for all $n \geq 4$, while lower bounds are proven for infinitely many values of $n$. Bounds written in bold are shown in this work.
    }
    \vspace{-3ex}
    \label{tab:forbidden_cells_summary}
\end{table}

\newpage
%%%%%%%%%%%%%%%%%%%%%%%%%%
%%     RELATED WORK     %%
%%%%%%%%%%%%%%%%%%%%%%%%%%
\subparagraph*{Related work.}

Apart from the aforementioned work of Ackerman and Tardos~\cite{ackerman_quasiplanar} and Kaufmann, Klemz, Knorr, Reddy, Schröder, and Ueckerdt~\cite{kaufmann_density_formula} on \threezero-free drawings, we are (as far as we know) the first to consider graph drawings with a forbidden type of cell.
However, some existing drawing styles are characterized by forbidding an infinite number of cells. These include planar and 1-planar drawings as well as the recently introduced $k^+$-real face drawings~\cite{binucciGraphsDrawnVertices2024}. 
Further, some beyond-planar drawing styles are characterized by forbidding a configuration that is somewhat similar to a particular cell type.
For example, in a fan-crossing-free drawing~\cite{cheongNumberEdgesFanCrossing2015} there is no edge crossed by two adjacent edges, which resembles the \fourone-cell.
Secondly, in a quasiplanar drawing, there are no three pairwise crossing edges, which resembles the \threezero-cell.
And in a $(2,2)$-grid-free drawing, there are no two pairs $e_1,e_2,f_1,f_2$ of edges such that every $e_i$ crosses every $f_j$, which resembles the \fourzero-cell.
For a collection of edge density results on these and other beyond-planar graph classes, we refer to \cite[Chapter 4]{didimoSurveyBeyondGraphDrawing2020}.
We remark that in the three cases above, every fan-crossing-free, quasiplanar, and $(2,2)$-grid-free drawing is also a \celltype-free drawing for $\celltype = \fourone$, \threezero, and \fourzero, respectively, while the converse statements are false.
Hence, lower bound constructions carry over to our setting, while in general, \celltype-free drawings might allow for a much higher edge density.

Let us also mention that cell types have been investigated in the context of arrangements of pseudolines.
For simple arrangements of pseudolines, it is known that only \threezero-cells are forced to appear in every sufficiently large arrangement.
In fact, any simple arrangement of $n$ pseudolines in the Euclidean plane or in the projective plane contains at least $n-2$ or $n$, respectively, \threezero-cells~\cite{leviTeilung1926,felsnerTrianglesEuclideanArrangements1999}.
On the other hand, there are various constructions without \fourzero-cells~\cite{harborthTWOCOLORINGSSIMPLEARRANGEMENTS1984,harborthSimpleArrangementsPseudolines1985}, and arbitrarily large simple arrangements consisting only of \threezero-cells and \fourzero-cells~\cite{leanosSimpleEuclideanArrangements2007}.
Further results on the number of cells of different types in pseudoline arrangements can be found in the textbook chapter~\cite{felsner2017pseudoline}, including upper bounds on the number of \threezero-cells and \fourzero-cells in simple pseudoline arrangements from~\cite{MR2858666,grunbaumArrangementsSpreads1972}, and a lower bound on the number of \fourzero-cells and \fivezero-cells from~\cite{roudneffQuadrilateralsPentagonsArrangements1987}.

%%%%%%%%%%%%%%%%%%%%%%%%%%
%%        OUTLINE       %%
%%%%%%%%%%%%%%%%%%%%%%%%%%
\subparagraph*{Outline.} 
We start in \cref{sec:prelims} with relevant definitions and some statements that we will use throughout the paper. 
Then, in \cref{sec:complete-graphs}, we show that most cell types can be avoided in simple drawings of $K_n$ and that all cell types can be avoided in arbitrary drawings of $K_n$.
Constituting the main body of this work, we address the remaining cell types in \crefrange{sec:fivezero}{sec:fourone}, with quasiplanarity discussed jointly with \threezero-cells in \cref{sec:threezero}.
Finally, in \cref{sec:all-graphs}, we consider the graph classes $\mathcal{G}_\celltype$ induced by forbidding cell types \celltype{} whose boundary does not contain any crossings in more detail. We conclude this work with some open problems in \cref{sec:conclusion}.

%%%%%%%%%%%%%%%%%%%%%%%%%%
%%                      %%
%%    PRELIMINARIES     %%
%%                      %%
%%%%%%%%%%%%%%%%%%%%%%%%%%
\section{Preliminaries: Drawings, Cells, and Cell Types} \label{sec:prelims}

Throughout this paper, we consider finite graphs with no loops but possibly parallel edges.
To make the (potential) presence of multiedges more explicit, we sometimes use the term \emph{multigraph}.
Similarly, when a graph has no parallel edges, it is called a \emph{simple} graph.
A drawing $\Gamma$ of a (potentially disconnected) multigraph $G=(V,E)$ maps vertices to pairwise distinct points and edges to curves connecting the corresponding vertices.
For convenience, we mostly consider drawings on the sphere $\mathbb{S}^2$ to avoid a special treatment of the unbounded cell.
As customary, we require that no edge passes through a vertex, any two edges have only finitely many points in common, each being a common endpoint or a proper crossing, and that no three edges cross in the same point.
A drawing $\Gamma$ is \emph{connected} if the image of the map $\Gamma$ is connected in $\mathbb{S}^2$.

A drawing is \emph{simple} if any two edges have at most one point in common, which is either a common endpoint or a crossing.
As this rules out parallel edges, only simple graphs admit simple drawings.
A \emph{lens} in a drawing is a region of $\mathbb{S}^2$ bounded by exactly two parts of edges.
These could be two parallel edges, two edges crossing twice, or two crossing adjacent edges.
We call a lens \emph{empty} if its interior contains neither a crossing nor a vertex.
A drawing is \emph{non-homotopic} if it does not contain any empty lens.

Consider any fixed connected drawing $\Gamma$ of some multigraph $G$.
The crossings split each edge into a number of \emph{edge segments}.
An \emph{outer edge segment} is incident to a vertex, while an \emph{inner edge segment} starts and ends with a crossing.
The \emph{cells} of $\Gamma$ are the connected components of $\mathbb{S}^2$ after the removal of all vertices and edges.
For a cell $c$, we denote by $v(c)$ the number of vertices, by $e(c)$ the number of edge segments, and by $e_{\rm in}(c)$ the number of inner edge segments on its boundary.
Here, vertices and edge segments may be counted multiple times; as often as they appear when walking along the boundary of $c$.
The \emph{size} of $c$ is the sum of the former two quantities: $\size{c} \coloneqq e(c) + v(c)$.
We denote by $\calC$ the set of all cells, and set $\calC_k = \{c \in \calC \colon \size{c} = k\}$ and $\calC_{\geq k} = \{c \in \calC \colon \size{c} \geq k\}$.

Since $\Gamma$ is assumed to be a connected drawing of $G$, the boundary of every cell is connected.
The \emph{type} of a cell $c$ is the cyclic sequence of incidences with edge segments, crossings, and vertices along the boundary of $c$.
We emphasize that a cell type \celltype{} only determines the form of incidence (edge segment, vertex, or crossing) and not the participating edge, vertex, or crossing of $G$.
In particular, different incidences might be with the same edge segment, crossing, or vertex.
We use pictograms with the size of the cell inscribed (such as \threezero{} or \fiveone{}) to concisely denote cell types of small size.
For convenience, 
we let these pictograms also denote the number of cells of that type in a given drawing.
If no cell in $\Gamma$ has type \celltype, then $\Gamma$ is called \emph{\celltype-free}.

Let us conclude the preliminaries with an observation by Kaufmann, Klemz, Knorr, Reddy, Schröder, and Ueckerdt~\cite{kaufmann_density_formula}, and two lemmas, which we will need in the course of this paper.

\begin{observation}[\hspace{1sp}{\cite[Observation 2.2]{kaufmann_density_formula}}]
    \label{obs:cell_types_non-hom}
    Let $\Gamma$ be any non-homotopic connected drawing of some multigraph $G$ with at least three vertices. Then
    \begin{itemize}
        \item $\calC_1 = \calC_2 = \emptyset$,
        \item $\calC_3$ is the set of all \threezero-cells,
        \item $\calC_4$ is the set of all \fourzero-cells and all \fourone-cells, and
        \item $\calC_5$ is the set of all \fivezero-cells, \fiveone-cells, and \fivetwo-cells.
    \end{itemize}
\end{observation}

The next two lemmata will help prove upper bounds for \fivezero-free drawings in \cref{sec:fivezero} and \threezero-free drawings in \cref{sec:threezero}. 
\cref{lem:identity_vertices_cell_sizes} below relates the number of vertices in a graph to the sizes of cells in a drawing of the graph.
The statement already appears implicitly in various edge density proofs via discharging (e.g., \cite{ackermanMaximumNumberEdges2009,arikushiGraphsThatAdmit2012,ackermanTopologicalGraphsMost2019}) and as a key ingredient in the proof for the Density Formula~\cite[Lemma 3.1]{kaufmann_density_formula}. 
We thus state it without proof\footnote{As one of many ways to prove it, for example, subtracting the Density Formula with $t=0$ from the Density Formula with $t=1$ gives the result.}.

\begin{lemma}
    \label{lem:identity_vertices_cell_sizes}
    For any connected drawing $\Gamma$ of a multigraph $G$ with at least one edge,
    \[
        |V(G)|-2 = \sum_{c \in \calC} \left(\frac{1}{4}\size{c} - 1\right).
    \]
\end{lemma}

The next lemma gives a bound on the number of \threezero-cells and \fourone-cells in terms of the number of inner edge segments in large cells of non-homotopic drawings. 
Similar statements have been implicitly proven in prior work \cite{ackerman_quasiplanar,ackermanTopologicalGraphsMost2019}. 

\begin{lemma}
\label{lem:corridor_proof}
    In any connected non-homotopic drawing of a multigraph with at least three vertices, it holds that 
    \[
        3 \cdot \threezero + \fourone \leq \sum_{c \in \calC_{\geq5}} e_{\rm in}(c).
    \]
\end{lemma}
\begin{proof}
    According to \cite[Lemma 3.6]{kaufmann_density_formula}, for the set $S_{\rm in}$ of all inner edge segments we have
    \begin{equation}
        |S_{\rm in}| \geq 3 \cdot \threezero + \fourone + 2 \cdot \fourzero.\label{eq:lemma3.6}
    \end{equation}
    On the other hand, every inner edge segment appears twice on the boundary of some cell.
    As the only cells of size at most four are \threezero-, \fourone-, and \fourzero-cells (cf.~\cref{obs:cell_types_non-hom}), it follows
    \[
        |S_{\rm in}| = \frac12 \cdot \sum_{c \in \calC} e_{\rm in}(c) = \frac12 \cdot \big(3 \cdot \threezero + \fourone + 4 \cdot \fourzero + \sum_{c \in \calC_{\geq5}} e_{\rm in}(c)\big),
    \]
    which together with \eqref{eq:lemma3.6} gives the claim.
\end{proof}

%%%%%%%%%%%%%%%%%%%%%%%%%%%%%%%%%%%%%%%%%%%%%%%
%%                                           %%
%% AVOIDABLE CELLS IN SIMPLE DRAWINGS OF K_N %%
%%                                           %%
%%%%%%%%%%%%%%%%%%%%%%%%%%%%%%%%%%%%%%%%%%%%%%%
\section{Drawings of Complete Graphs}
\label{sec:complete-graphs}

First, we show that for most cell types \celltype, complete graphs admit simple \celltype-free drawings.

\begin{construction}
    \label{con:most_cell_types_avoidable}
    For every $n\geq 4$, the complete graph $K_n$ admits a simple \fiveone-free drawing in which each cell $c$ fulfills $\size{c} \leq 5$ or $\size{c} = 2n$.
\end{construction}
\begin{proof}
    Consider the simple drawing of $K_n$ depicted in \cref{fig:three_bend_Kn}, in which the vertices lie evenly spaced on a horizontal line and each edge follows the shape of a $\wedge$ with a right-angled peak.
    For $n \geq 4$, all cells in this drawing except the outer cell are \threezero-, \fourzero-, \fourone-, \fivezero-, or \fivetwo-cells.
    Clearly, the outer cell $c$ has size $\size{c} = 2n$.
\end{proof}

\begin{figure}[htb]
    \centering
    \includegraphics{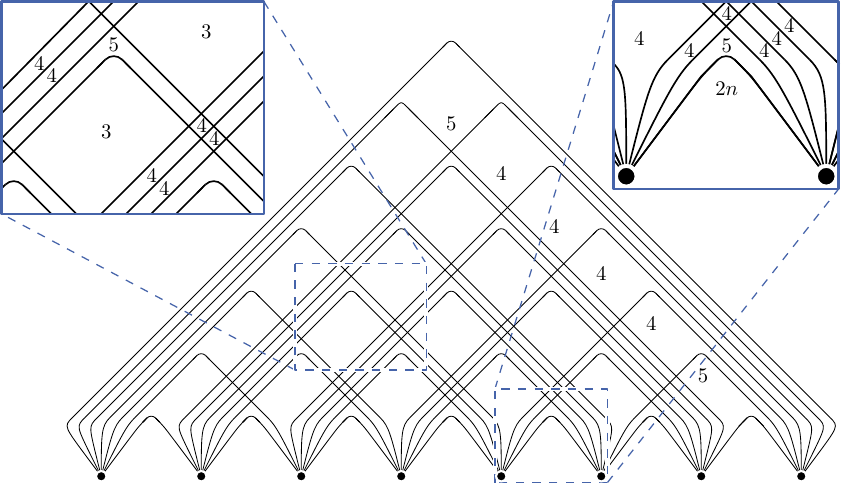}
    
    \caption{A simple drawing of $K_n$ (here $n=8$) containing only cells of size at most~$5$, and one cell of size $2n$. The numbers in the figure indicate the sizes of the different kinds of cells.}
    \label{fig:three_bend_Kn}
\end{figure}

\cref{con:most_cell_types_avoidable} shows that for any cell type $\celltype \notin \{\threezero, \fourzero, \fourone, \fivezero, \fivetwo\}$, every large enough complete graph admits a simple $\celltype$-free drawing.
In particular, the edge density of simple $\celltype$-free drawings is $\binom{n}{2}$ whenever $\celltype$ has size at least~$6$, or $\celltype = \fiveone$.
\
Next, we show that \fivetwo{}-cells can also be avoided in simple drawings of $K_n$. 
Note that any drawing without uncrossed edges will do, see for example \cite{harborthEdgesWithoutCrossings1974}.
For completeness, we present our own construction.

\begin{construction}
    \label{con:no-fivetwo-cell-complete-graphs}
    For every $n \geq 8$, the complete graph $K_n$ has a simple \fivetwo-free drawing.
\end{construction}
\begin{proof}
    Let $n \geq 8$ be even and start with a straight-line drawing of $K_n$ with vertices as equidistant points on a circle.
    By the \emph{length} of an edge we refer to the distance of its endpoints in circular order along the circle.  
    Redraw the edges of length~$3$ as depicted in blue in \cref{fig:almost_convex_Kn}, so as to cross the edges of length~$1$ (on the convex hull).
    Note that this change retains simplicity.
    Further, redraw every other edge of length~$2$ to be uncrossed in the outer cell (red edges in \cref{fig:almost_convex_Kn}).
    Finally, redraw all edges of length~$n/2$ to be pairwise crossing in the outer cell (green edges in \cref{fig:almost_convex_Kn}).
    This again retains simplicity.
    For $n$ odd, take the construction for $n'=n-1$, add the last vertex in the central cell, and connect it to all other vertices with straight-line segments.
    The resulting drawing $\Gamma$ is simple. Further, since every edge is crossed at least once, $\Gamma$ is also \fivetwo-free .
\end{proof}

\begin{figure}[htb]
    \centering
        \includegraphics{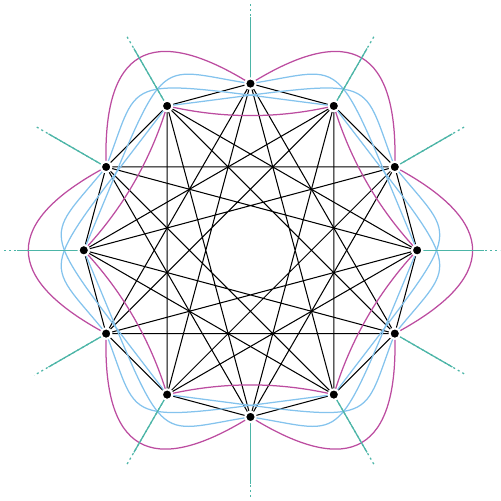}
        \caption{
			A simple \fivetwo-free drawing of $K_{12}$. (A small perturbation eliminates the multi-crossings.) Edges of length three are drawn in blue and those of length two in red. Edges of length $n/2$ are depicted in green and only partially drawn. 
			}
    \label{fig:almost_convex_Kn}
\end{figure}

%%%%%%%%%%%%%%%%%%%%%%%%%%%%%%
%%                          %%
%%    ARBITRARY DRAWINGS    %%
%%                          %%
%%%%%%%%%%%%%%%%%%%%%%%%%%%%%%
With \cref{con:most_cell_types_avoidable,con:no-fivetwo-cell-complete-graphs} at hand, only the cell types \threezero, \fourone, \fourzero, and \fivezero{} remain, if we restrict to simple drawings of complete graphs.
Before turning to these types, let us close the section by dropping the requirement that the drawing is simple.
In fact, without any restriction, every cell type can be avoided in drawings of complete graphs.

\begin{construction}
    \label{con:arbitrary_quadratic_density}
    For any cell type \celltype and any $n \geq 8$, the complete graph $K_n$ admits a \celltype-free drawing.
\end{construction}
\begin{proof}
    Let $n \geq 8$ be arbitrary. 
    For $\celltype \notin \{\threezero, \fourzero, \fourone, \fivezero\}$, the claim follows from \cref{con:most_cell_types_avoidable,con:no-fivetwo-cell-complete-graphs}.
    Every remaining cell type \celltype{} has size $3$, $4$, or $5$ and at least one crossing on its boundary.
    We now construct a drawing of $K_n$ in which every cell with an incident crossing has size $2$ or at least $6$, which completes the proof.
    To do so, start with any drawing of $K_n$.
    
    Locally replace each crossing with the pattern in \cref{fig:crossing-replacement}.
    In the resulting drawing, each cell incident to a crossing is part of such a pattern and therefore of size $2$ or at least~$6$. 
\end{proof}

\begin{figure}[htb]
    \centering
    \includegraphics{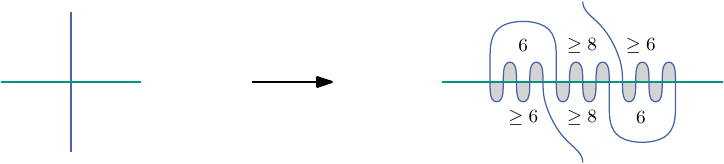}
    \caption{
        A crossing can be redrawn to bound only cells of size 2 (shaded gray) or at least~$6$.
    }
    \label{fig:crossing-replacement}
\end{figure}

%%%%%%%%%%%%%%%%%%%%%%%%%%%%%%%%%%
%%                              %%
%%  NON-HOMOTOPIC 4_0/5_0 CELL  %%
%%                              %%
%%%%%%%%%%%%%%%%%%%%%%%%%%%%%%%%%%
\section{Drawings without \texorpdfstring{\fourzero}{4\_0} or \texorpdfstring{\fivezero}{5\_0}-Cells}
\label{sec:fivezero}
\label{sec:fourzero}

In this section, we consider the \fourzero- and \fivezero-cell.
We start by showing that non-homotopic \fivezero-free drawings have only linear edge density.
Note that by \cref{con:arbitrary_quadratic_density}, the restriction to non-homotopic drawings is essential.
This result can also be derived from the proof of a statement on straight-line drawings with crossing angle restrictions~\cite[Theorem 4.2]{dujmovicLargeAngleCrossing2011}.

\begin{theorem}
    \label{thm:nh_fivezero_free}
    For $n \geq 3$, any connected $n$-vertex multigraph that admits a non-homotopic \fivezero-free drawing has at most $6n-12$ edges.
\end{theorem}
\begin{proof}
    Fix a \fivezero-free drawing of a connected $n$-vertex multigraph with $n \geq 3$.
    First, let us compute for each cell $c \in \calC$ the quantity $3e(c) + 2v(c) - 12$, and in case $c \in \calC_{\geq 5}$ compare it with $e_{\rm in}(c)$.
    \begin{align*}
        &c\text{ is a \threezero-cell}:& &3e(c) + 2v(c) - 12 = -3\\
        &c\text{ is a \fourone-cell}:& &3e(c) + 2v(c) - 12 = -1\\
        &c\text{ is a \fourzero-cell}:& &3e(c) + 2v(c) - 12 = 0\\
        &c\text{ is a \fiveone-cell}:& &3e(c) + 2v(c) - 12 = 2 = e_{\rm in}(c)\\
        &c\text{ is a \fivetwo-cell}:& &3e(c) + 2v(c) - 12 = 1 \geq 0 = e_{\rm in}(c)\\
        &\size{c} \geq 6:& &3e(c) + 2v(c) - 12 = e(c) + 2\size{c} - 12 \geq e(c) \geq e_{\rm in}(c)
    \end{align*}
    Note that the above case distinction is complete as $c$ is not a \fivezero-cell.
    
    Summing over all cells, this gives
    \begin{equation}
        \sum_{c \in \calC} (3e(c)+2v(c)-12) \geq -3 \cdot \threezero - \fourone + \sum_{c \in \calC_{\geq 5}} e_{\rm in}(c) \geq 0,\label{eq:fivezero_inequality}
    \end{equation}
    where the last inequality is an application of \cref{lem:corridor_proof}.

    Finally, we obtain the desired result by double-counting all vertex-cell incidences as
    \begin{equation*}
        2|E| = \sum_{c \in \calC}v(c) \overset{\eqref{eq:fivezero_inequality}}{\leq} \sum_{c \in \calC}\left(3e(c) + 3v(c) - 12 \right) = 3\sum_{c \in \calC} \left(\size{c}-4\right) = 12(n-2),
    \end{equation*}
    where the last equality is an application of \cref{lem:identity_vertices_cell_sizes}.
\end{proof}

\cref{thm:nh_fivezero_free} gives an upper bound of $6n-12$ on the edge density of \fivezero-free non-homotopic drawings of multigraphs.
As the next construction shows, this bound is tight, even for simple drawings.
This construction also provides our best lower bound for the edge density of simple $\fourzero$-free drawings.

\begin{construction}
    \label{con:dense_fivezero_free}
    \label{con:dense_fourzero_free}
    For infinitely many values of $n$, there is a simple $n$-vertex graph with $6n-12$ edges that admits a simple $\fourzero$-free and $\fivezero$-free drawing.
\end{construction}
\begin{proof}
    Start with a plane drawing $\Gamma_0$ of a $5$-connected triangulation $G_0$ on $n$ vertices.
    For an edge $e = uv$ in $G_0$, consider the two triangular faces $f,f'$ in $\Gamma_0$ incident to $e$, and let $w,w'$ be the third vertex (different from $u,v$) in $f,f'$, respectively.
    We insert a new edge between $w$ and $w'$ in the drawing by only traversing the two incident faces $f,f'$ of $e$.
    As $G_0$ is $5$-connected, $ww'$ is not already an edge in $G_0$, as otherwise there is a separating triangle in $G_0$ with $w,w'$ and one of $u,v$.
    Now, we do the same for all edges in $G_0$.
    Again, as $G_0$ is $5$-connected, for every edge $\tilde{e} \neq e$ in $G_0$ we obtain a pair different from $w,w'$, as otherwise there is a separating $4$-cycle in $G_0$ with $w,w'$ and one endpoint of each of $e,\tilde{e}$.

    \begin{figure}[htb]
        \centering
        \begin{subfigure}{.48\textwidth}
        \centering
        \includegraphics[align=c]{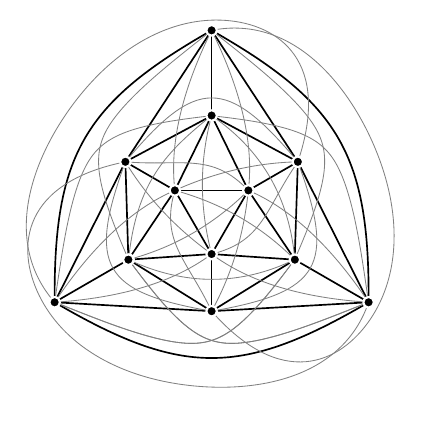}
        \end{subfigure}
        \hfill
        \begin{subfigure}{.48\textwidth}
            \centering
        \includegraphics[align=c]{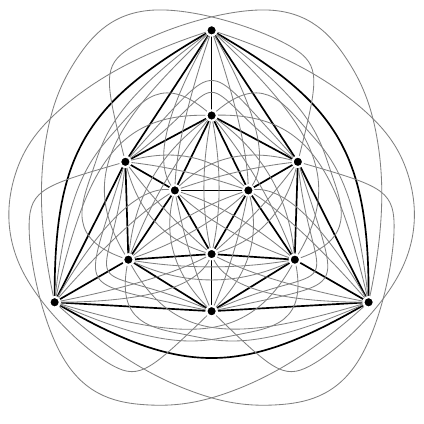}
        \end{subfigure}
        \caption{
            Left: A simple \fourzero-free and \fivezero-free drawing with $6n-12$ edges.
            Right: A non-homotopic \fourzero-free drawing of a multigraph with $9n-18$ edges. On both sides, $n=12$.
        }
        \label{fig:dense_fivezero_free}
        \label{fig:dense_fourzero_free}
        
    \end{figure}

    Call the resulting graph $G$ and the resulting drawing $\Gamma$.
    By the above reasoning, the graph $G$ is simple (has no parallel edges).
    Moreover, observe that the drawing $\Gamma$ is simple (any two edges have at most one point in common).
    See \cref{fig:dense_fivezero_free} (Left) for an illustration.

    The only cells in $\Gamma$ that are not incident to any vertex are \threezero-cells.
    Hence, $\Gamma$ is \fourzero-free and \fivezero-free.
    Finally, since $G_0$ is a plane triangulation, we have $|E(G_0)| = 3n - 6$ and thus $|E(G)| = 2 |E(G_0)| = 6n-12$.
\end{proof}

A slight variation of this construction yields a better bound for non-homotopic $\fourzero$-free drawings.

\begin{construction}
    \label{con:dense_multi_fourzero_free}
    For infinitely many values of $n$, there is an $n$-vertex multigraph with $9n-18$ edges that admits a non-homotopic \fourzero-free drawing.
\end{construction}
\begin{proof}
    As in \cref{con:dense_fivezero_free}, we start with any $5$-connected triangulation $G_0$ on $n$ vertices.
    For each edge $e = uv$ in $G_0$, we again consider the unique pair $w,w'$ of common neighbors of $u$ and $v$.
    But this time, we add two parallel edges between $w$ and $w'$, again drawing these new edges so that they only cross the edge $e$ of $G_0$.
    For each triangle $uvw$ of $G_0$, this yields two edges per vertex that emanate from the vertices into the triangle.
    
    Additionally, we draw the new edges in a way that in every triangular face $uvw$ of $G_0$, the two parallel edges from each vertex enclose exactly two crossings of the four other edges that emanate into that face.
    This way, $uvw$ is split into one cell of size $6$ with no incident vertex and six \threezero-cells in the interior of $uvw$, as well as one \fivezero-cell at each edge of $uvw$ and three cells at each vertex of $uvw$ (one \fiveone-cell and two \fourone-cells); see \cref{fig:dense_fourzero_free} (right).

    Thus, the obtained drawing is \fourzero-free.
    Non-homotopicity is apparent since only the pairs of parallel edges form lenses, which are non-empty.
    Finally, the drawing contains $9n-12$ edges, since for each of the $3n-6$ edges of $G_0$, two further edges are added.
\end{proof}

By \cref{con:dense_multi_fourzero_free}, the edge density of \fourzero-free non-homotopic multigraphs is at least $9n-18$.
For simple \fourzero-free graphs, the edge density is at least $6n-12$ by \cref{con:dense_fourzero_free}.
However, we have no non-trivial upper bound; see also \cref{prob:fourzero_density} in \cref{sec:conclusion}.

%%%%%%%%%%%%%%%%%%%%%%%%%%%%%%
%%                          %%
%%  NON-HOMOTOPIC 3_0 CELL  %%
%%                          %%
%%%%%%%%%%%%%%%%%%%%%%%%%%%%%%
\section{Quasiplanar and \texorpdfstring{\threezero}{3\_0}-Free Drawings}
\label{sec:threezero}

As mentioned already, \threezero-free drawings are closely related to \emph{quasiplanar} drawings, that is, drawings in which no three edges cross pairwise.
As self-intersection of edges is forbidden, any quasiplanar drawing is \threezero-free but there exist \threezero-free drawings that are not quasiplanar.

Several results from the existing literature contain upper and lower bounds on the edge density of \threezero-free and quasiplanar drawings.
See \cref{tab:quasiplanar_summary} for an overview.
In fact, the edge density for both drawing styles is known to be $8n-20$ in the case of non-homotopic drawings of multigraphs.
However, there are still some gaps in the case of non-homotopic drawings of simple graphs, and the case of simple drawings of (simple) graphs.
In this section, we provide several improvements.

\setlength{\tabcolsep}{10pt}
\begin{table}[htb]
    \centering
    \begin{tabular}{c|c|cc|cc}
        \toprule
        & Drawing Type & \multicolumn{2}{c|}{Lower Bounds} & \multicolumn{2}{c}{Upper Bounds} \\
        \midrule
        \parbox[t]{4mm}{\multirow{6}{*}{\rotatebox[origin=c]{90}{\ quasiplanar}}} & \customstack{Non-homotopic}{multigraph} &  
            \multicolumn{2}{c|}{\customstack{$8n-20$}{\cite[Thm. 6.7]{kaufmann_density_formula}}} & \multicolumn{2}{c}{\customstack{$8n-20$}{\cite[Thm. 3]{ackerman_quasiplanar}}}\\
        \cmidrule{2-6}
            & \customstack{Non-homotopic}{simple graph} &  
            \customstack{$7n-28$}{\cite[Thm. 6.7]{kaufmann_density_formula} (implicit)} & \customstack{$7.5n-28$}{\cref{con:simple_nh_quasiplanar}} &
            \multicolumn{2}{c}{\customstack{$8n-20$}{\cite[Thm. 1]{ackerman_quasiplanar}}}\\
        \cmidrule{2-6}
            & Simple drawing & 
            \multicolumn{2}{c|}{\customstack{$6.5n-20$}{\cite[Thm. 6.11]{kaufmann_density_formula}}} &
            \multicolumn{2}{c}{\customstack{$6.5n-20$}{\cite[Thm. 5]{ackerman_quasiplanar}}}\\
        \midrule
        \parbox[t]{4mm}{\multirow{6}{*}{\rotatebox[origin=c]{90}{\threezero-free}}} & \customstack{Non-homotopic}{multigraph} &  
            \multicolumn{2}{c|}{\customstack{$8n-20$}{\cite[Thm. 6.7]{kaufmann_density_formula}}} &
            \multicolumn{2}{c}{\customstack{$8n-20$}{\cite[Thm. 3]{ackerman_quasiplanar}}}\\
        \cmidrule{2-6}
            & \customstack{Non-homotopic}{simple graph} &  
            \customstack{$7.5n-\mathcal{O}(1)$}{\cite[Thm. 4]{ackerman_quasiplanar} (implicit)} & \customstack{$8n-28$}{\cref{con:simple_nh_threezero_free}} &
            \multicolumn{2}{c}{\customstack{$8n-20$}{\cite[Thm. 3]{ackerman_quasiplanar}}}\\
        \cmidrule{2-6}
            & Simple drawing & 
            \multicolumn{2}{c|}{\customstack{$7n-30$}{\cref{con:no_threezero_simple_dense} based on \cite[Thm. 4]{ackerman_quasiplanar}}} &
            \multicolumn{2}{c}{\customstack{$7n-20$}{\cite[Thm. 6]{ackerman_quasiplanar}}}\\
        \bottomrule
    \end{tabular}
    \medskip
    \caption{
        Previous and new upper and lower bounds on the number of edges in $n$-vertex quasiplanar drawings and $n$-vertex \threezero-free drawings. Upper bounds hold for all $n \geq 4$, while lower bounds are proven for infinitely many values of $n$.
    }
    \label{tab:quasiplanar_summary}
\end{table}

For quasiplanar drawings, we improve the lower bound for non-homotopic drawings of simple graphs from $7n-28$~\cite{kaufmann_density_formula} to $7.5n-28$ (\cref{con:simple_nh_quasiplanar}).
We remark that there still remains a gap to the upper bound of $8n-20$~\cite{ackerman_quasiplanar}.

For \threezero-free drawings, we improve the lower bound for non-homotopic drawings of simple graphs from $7.5n - \mathcal{O}(1)$~\cite{ackerman_quasiplanar} to $8n-28$ (\cref{con:simple_nh_threezero_free}), almost matching the upper bound of $8n-20$~\cite{ackerman_quasiplanar}.

For \threezero-free simple drawings, we give a lower bound of $7n-30$ (\cref{con:no_threezero_simple_dense}) based on a construction of Ackerman and Tardos , who also proved an upper bound of $7n-20$ \cite[Theorem 6]{ackerman_quasiplanar}.
See again \cref{tab:quasiplanar_summary}.

%%%%%%%%%%%%%%%%%%%%%%%%%%%
%%      LOWER BOUND      %%
%%%%%%%%%%%%%%%%%%%%%%%%%%%
\subsection{Lower~Bounds~for~\texorpdfstring{\threezero}{3\_0}-Free~and~Quasiplanar Drawings~of~Simple~Graphs}
\label{subsec:lower_bounds}

We start with two constructions for non-homotopic drawings of simple graphs.
First, a quasiplanar construction, followed by a construction of \threezero-free drawings.
\begin{construction}
\label{con:simple_nh_quasiplanar}
    For infinitely many values of $n$, there is a simple $n$-vertex graph with $7.5n - 28$ edges that admits a non-homotopic quasiplanar drawing.
\end{construction}
\begin{proof}
    \crefname{enumi}{step}{steps}
    Let $n \geq 14$ be even.
    We describe a drawing on $n$ vertices with the desired properties.
    Precise edge paths are obtained by avoiding unnecessary crossings.  
    \cref{fig:simple_nh_quasiplanar} (top) depicts the final drawing $\Gamma$. 
    We draw
    \begin{enumerate}
        \item an $\frac{n}{2}$-cycle $C_{\rm in}$ as a circle with equidistant points and a concentric copy $C_{\rm out}$ of $C_{\rm in}$ with larger diameter. \label{enum:quasiplanar1}
    \end{enumerate}

        We call two radially aligned vertices on $C_{\rm in}$ and $C_{\rm out}$ \emph{partners}.
        For any vertex $x$, we denote by~$x'$ its partner, by $x_{+k}$ the vertex reached by $k$ clockwise steps on the cycle of $x$, and by~$x'_{+k}$ the partner of $x_{+k}$.
        We continue drawing
    \begin{enumerate}
        \setcounter{enumi}{1}
        \item a straight edge between any pair of partners, \label{enum:quasiplanar2}
        \item for any vertex $v$, a straight edge to $v'_{+1}$, \label{enum:quasiplanar3}
        \item for any vertex $v$ on $C_{\rm in}$ ($C_{\rm out}$) an edge from $v$ to $v_{+2}$, drawn through the interior of $C_{\rm in}$ (exterior of $C_{\rm out}$),
        \item for any vertex $v$ on $C_{\rm in}$ ($C_{\rm out}$), an edge to $v'_{+2}$. The edge passes through the space between $C_{\rm in}$ and $C_{\rm out}$, crosses $C_{\rm out}$ ($C_{\rm in}$) in the edge $v' v'_{+1}$ and continues through the exterior of $C_{\rm out}$ (interior of $C_{\rm in}$),
        \item for any vertex $v$ on $C_{\rm in}$ ($C_{\rm out}$) an edge to $v_{+3}$. The edge first passes through the space between $C_{\rm in}$ and $C_{\rm out}$, crosses $C_{\rm in}$ ($C_{\rm out}$) between $v_{+1}$ and $v_{+2}$ and then goes through the interior of $C_{\rm in}$ (exterior of $C_{\rm out}$), \label{enum:quasiplanar6}
        \item two zig-zag paths inside of $C_{\rm in}$ (outside of $C_{\rm out}$), whose edges connect vertices of distance at least 4 on $C_{\rm in}$ ($C_{\rm out}$) such that no parallel edges are created.
    \end{enumerate}

\begin{figure}[!htb]
    \centering
    \includegraphics{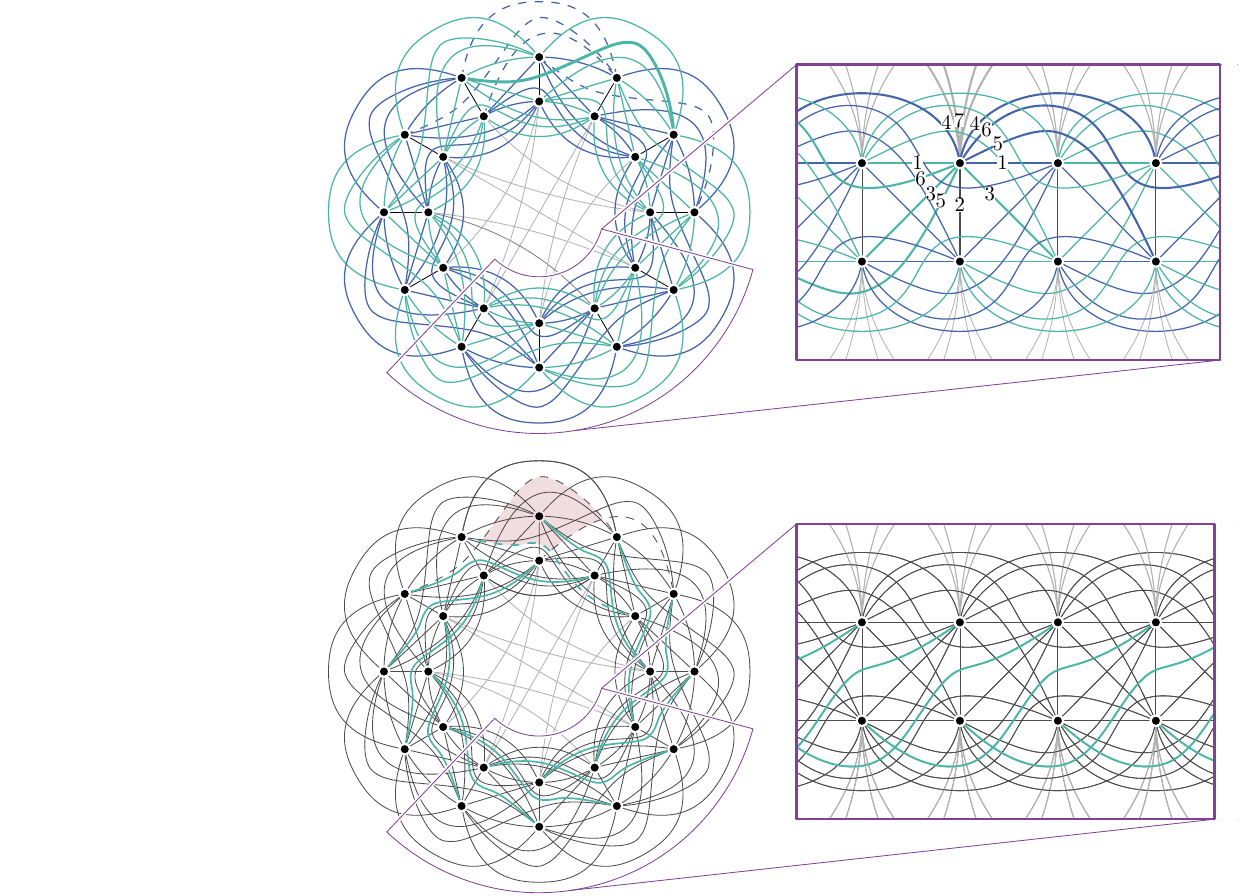}
    \caption{
        The drawings of \cref{con:simple_nh_quasiplanar,con:simple_nh_threezero_free} for $n=24$ with the outer zig-zag paths omitted for visual clarity, as well as zoomed-in strips of the local edge patterns.\\
        Top: \cref{con:simple_nh_quasiplanar} where edges except on the inner zig-zag paths are partitioned by color into three plane drawings.
        In the zoomed-in strip, the incident edges of one vertex are marked by the step of the construction, in which they are added. 
        In the The edges that form lenses with the thick edge are dashed.\\
        Bottom: \cref{con:simple_nh_threezero_free} where green edges are those not present in \cref{con:simple_nh_quasiplanar}.
        The shaded area is bounded by three pairwise crossing edges (drawn dashed).
    }
    \label{fig:simple_nh_quasiplanar}
    \label{fig:nh_threezero_free}
\end{figure}
    Step \ref{enum:quasiplanar1} adds $n$ edges, Step \ref{enum:quasiplanar2} adds $\frac{n}{2}$ edges, Steps \ref{enum:quasiplanar3}-\ref{enum:quasiplanar6} add $n$ edges each, and the final step contributes another $4 \cdot(\frac{n}{2} - 7)$ edges, totaling the desired $7.5n-28$ edges. 
    As the various steps connect vertices at different distances, the underlying graph is simple.

    We now argue that $\Gamma$ is quasiplanar. First, consider all edges except those in zig-zag paths.
    These can be locally 3-colored such that one color class consists only of the edges in \cref{enum:quasiplanar2} and edges of the same color do not cross (see \cref{fig:simple_nh_quasiplanar}).
    Therefore, any 3 pairwise crossing edges would contain an edge from \cref{enum:quasiplanar2}, but each such edge is only crossed by two other edges, which in turn do not cross.

    Now, consider without loss of generality the zig-zag paths $P$ and $P'$ in the interior of $C_{\rm in}$ and some edge $e=uv$ in $P$.
    Observe that $e$ is crossed by three sets of edges:
    (i) Edges in $P'$,
    (ii) edges from Steps 1-10 that cross $e$ near $u$, and
    (iii) edges from Steps 1-10 that cross $e$ near $v$. 
    
    We argue that no edges from (i), (ii), or (iii) cross. Note that no pair of edges in the same group crosses.
    Next, an edge from (ii) cannot cross an edge from (iii) since $u$ and $v$ have distance at least~4 on~$C_{\rm in}$, which two edges cannot span while crossing. 
    Further, note that no edge in (i) is incident to $u$ or $v$. Since edges in (ii) and (iii) cross edges from zig-zag paths only at $u$ and $v$, they cannot cross edges from (i).
    Thus, $e$ is not part of three pairwise crossing edges, and in total, the drawing is quasiplanar.

    Finally, we argue that $\Gamma$ is non-homotopic. Note that in the above step-by-step construction, simplicity is only violated by the edges added in \cref{enum:quasiplanar6}. 
    Each such edge is part of four lenses (see dashed edges in \cref{fig:simple_nh_quasiplanar} (top)).
    As each of these lenses contains a vertex, the drawing is non-homotopic.
\end{proof}

\begin{construction}
\label{con:simple_nh_threezero_free}
    For infinitely many values of $n$, there is a simple graph on $n$ vertices and $8n - 28$ edges that admits a non-homotopic drawing without \threezero-cells.
\end{construction}
\begin{proof}
    Let $n \geq 14$ be even and start with the non-homotopic quasiplanar and hence also \threezero-free drawing from \cref{con:simple_nh_quasiplanar}.
    We will reuse the notation established there.
    For any vertex $v$ on~$C_{\rm out}$, we draw an edge from~$v$ to~$v'_{+3}$ that first moves through the space between $C_{\rm in}$ and $C_{\rm out}$, crosses $C_{\rm in}$ in $v'_{+1} v'_{+2}$, and then passes through the interior of $C_{\rm in}$.  
    This process adds $\frac{n}{2}$ edges, none of which are parallel to each other or previous edges. 
    Therefore, we obtain a drawing of a simple graph with the claimed density.
    The resulting $\Gamma$ is shown in \cref{fig:nh_threezero_free} (Bottom).

    Note that $\Gamma$ is not quasiplanar as it contains triples of pairwise crossing edges (see again \cref{fig:nh_threezero_free} (Bottom)).
    However, checking the cells along the paths of the newly added edges, it can be verified that the construction preserves \threezero-freeness. 
    Similarly, one can check that none of the lenses introduced by the new edges is empty, and so $\Gamma$ also remains non-homotopic. 
\end{proof}

Finally, we consider simple \threezero-free drawings.
In this setting, Ackerman and Tardos mention that a $7n - \mathcal{O}(1)$ lower bound can be obtained by taking a subdrawing of their construction in \cite[Theorem 4]{ackerman_quasiplanar}.
However, they neither specify nor optimize the constant term.
Doing so, we obtain a lower bound of $7n-30$, almost matching the $7n-20$ upper bound by Ackerman and Tardos.

\begin{construction}
\label{con:no_threezero_simple_dense}
    For infinitely many values of $n$, there is a simple $n$-vertex graph with $7n - 30$ edges that admits a simple \threezero-free drawing. 
\end{construction}
\begin{proof}
    Fix $m \geq 3$ and start with an $m \times 3$-grid of hexagons on $8m + 6$ vertices, each hexagon containing a straight-line drawing of $K_6 \setminus e$ where the vertical diagonal is missing (see \cref{fig:no_threezero_simple_dense}, left).
    Identify opposite vertices along the side of the grid of length $m$ to obtain the surface of a cylinder whose top and bottom boundary contain six vertices each.
    This leaves us with $n \coloneqq 6m + 6$ vertices, which we group into $2m+2$ layers of three vertices each, labeled from top to bottom, see \cref{fig:no_threezero_simple_dense}, middle.
    The vertices on the boundary are those in layers $1$ and $2$ (top boundary), and $2m+1$ and $2m+2$ (bottom boundary).
    All vertices in layers $3,\ldots,2m$ currently have degree $11$, while the vertices in layers $2$ and $2m+1$ have degree $9$ and those in layers $1$ and $2m+2$ degree $4$.
    Note that, so far, the drawing is \threezero-free.
    
    Now, add long edges as depicted in \cref{fig:no_threezero_simple_dense}, middle.
    Every vertex gains one edge to a neighboring layer and one edge each to a layer three above and below (if those layers exist).
    
    This increases the degree of the vertices on layers $4,\ldots,2m-1$ by $3$ and the remaining vertex degrees by $2$. 
    After adding a triangle on layer $1$ and $2m+2$, we arrive at the final drawing $\Gamma$ (see \cref{fig:no_threezero_simple_dense}, right).
    Grouping layers with the same vertex degrees, we find that the number of edges in~$\Gamma$ is
    \[
        \frac12\Big( (2m-4)\cdot3\cdot14+2\cdot3\cdot(13+11+7) \Big) = 42m+12 = 7n-30.
    \]

    \begin{figure}[htb]
        \centering
        \begin{subfigure}[t]{0.28\textwidth}
            \centering
            \includegraphics[scale=.95,page=1]{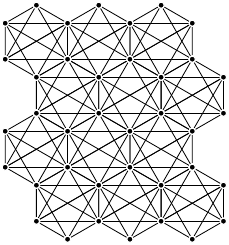}
            \label{subfig:no_threezero_simple_dense_a}
        \end{subfigure}%
        \begin{subfigure}[t]{0.38\textwidth}
            \centering
            \includegraphics[scale=.95,page=4]{no-3_0-cell-simple-lb.pdf}
            \label{subfig:no_threezero_simple_dense_b}
        \end{subfigure}
        \begin{subfigure}[t]{0.33\textwidth}
            \centering
            \includegraphics[scale=.95,page=3]{no-3_0-cell-simple-lb.pdf}
            \label{subfig:no_threezero_simple_dense_c}
        \end{subfigure}
         \caption{Illustrations of \cref{con:no_threezero_simple_dense} for $m=4$.}
        \label{fig:no_threezero_simple_dense}
    \end{figure}

    Note that no $\threezero$-cells are created by adding the long edges and the two triangles.
    Finally, the simplicity of $\Gamma$ follows from the fact that $\Gamma$ is locally homeomorphic to a straight-line drawing and no two edges are long enough to meet twice on the cylinder.
\end{proof}

%%%%%%%%%%%%%%%%%%%%%%%%%%%%%%%%%%%%%%%%
%%                                    %%
%%      3_0-FREE NON QUASIPLANAR      %%
%%                                    %%
%%%%%%%%%%%%%%%%%%%%%%%%%%%%%%%%%%%%%%%%

\subsection{Simple Non-Quasiplanar \texorpdfstring{\threezero}{3\_0}-Free Drawings}
\label{subsec:nonquasiplanar}

\cref{tab:quasiplanar_summary} summarizes bounds on the edge densities of six classes of graphs defined by admitting either quasiplanar or \threezero-free drawings.
The differing edge density bounds and the fact that some of the classes contain multigraphs while others do not imply that most of these graph classes are indeed distinct.
However, for non-homotopic quasiplanar drawings and non-homotopic \threezero-free drawings, their edge density is not sufficient for a separation result. 
Because of this, we asked in the conference version of the present paper \cite[Open Problem 2]{hahn_edge_densities_gd}, whether every graph that admits a non-homotopic \threezero-free drawing also admits a non-homotopic quasiplanar drawing.
Here, we show that this is not the case.
For this, we use that every subdrawing of a quasiplanar drawing is also quasiplanar.

\begin{proposition}\label{con:non-quasi-3_0-free}
    There are infinitely many graphs that admit a simple \threezero-free drawing but no non-homotopic quasiplanar drawing.
\end{proposition}
\begin{proof}
    Start with an arbitrary graph $G$ that does not admit a non-homotopic quasiplanar drawing (any graph on $n$ vertices and $> 8n-20$ edges suffices).
    Let $\Gamma$ be any simple drawing of $G$.
    If $\Gamma$ contains no \threezero-cells, we are done.
    Otherwise, let $c$ be a \threezero-cell and $e=uv$ one of the edges on its boundary.
    Add a new vertex $w$ in $c$ to $\Gamma$ and draw an edge $uw$ that starts at $u$, follows $e$ closely (such that $uw$ crosses the exact same edges as $e$ does between $u$ and the part where $e$ bounds the \threezero-cell containing $w$, in the same order) until entering $c$ and ending in $w$; see \cref{fig:non-quasi-3_0-free}.
    In the new drawing $\Gamma^+$, the corridor bounded by $e$ and $uw$ only contains a \fourone-cell and some \fourzero-cells.
    Moreover, $c$ becomes a (degenerate) \sevenone-cell in $\Gamma^+$.
    All other cells in $\Gamma^+$ incident to $uw$ have a corresponding cell of the same type in $\Gamma$ with the only difference that the edge segments of $e$ are replaced by edge segments of $uw$.
    Therefore, adding $w$ and $uw$ reduces the number of \threezero-cells.
    Further, the new drawing $\Gamma^+$ retains simplicity, as $uw$ crosses a subset of the edges crossed by $e$.
    
    Repeating this process, a simple \threezero-free drawing $\Gamma'$ of a graph $G'$ with $G \subseteq G'$ is obtained.
    Since $G$ does not admit a non-homotopic quasiplanar drawing, neither does $G'$, and so, $G'$ is as desired.
\end{proof}

\begin{figure}[htb]
    \centering
    \includegraphics{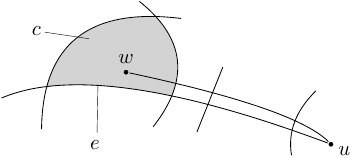}
    \caption{Illustration of \cref{con:non-quasi-3_0-free}. By adding a vertex and an edge, the number of \threezero-cells in a drawing can be reduced by one.}
    \label{fig:non-quasi-3_0-free}
\end{figure}

%%%%%%%%%%%%%%%%%%%%%%%%%%%%%%
%%                          %%
%%  NON-HOMOTOPIC 4_1 CELL  %%
%%                          %%
%%%%%%%%%%%%%%%%%%%%%%%%%%%%%%
\section{Drawings without \texorpdfstring{\fourone}{4\_1}-Cells}
\label{sec:fourone}

In this section, we consider the \fourone-cell.
In the conference version of the present paper \cite[Construction 9]{hahn_edge_densities_gd}, we presented non-homotopic \fourone-free drawings of $n$-vertex graphs with $\Omega(n^2)$ edges. There, we erroneously claimed that the resulting drawings are simple.
In the present paper, this is \cref{con:fourone-non-hom}.
For simple drawings, we were so far unable to construct \fourone-free drawings in the plane with a quadratic number of edges.
Still, we do provide a construction with a superlinear number of edges (\cref{con:fourone_simple_plane}).
Curiously, we \emph{can} achieve a quadratic edge density even for simple drawings when we consider drawings on the torus instead of the plane (\cref{con:fourone_simple_torus}).

%%%%%%%%%%%%%%%%%%%%%%%%%%%%%%
%%                          %%
%%     SIMPLE DRAWINGS      %%
%%                          %%
%%%%%%%%%%%%%%%%%%%%%%%%%%%%%%

\subsection{Simple \texorpdfstring{\fourone}{4\_1}-Free Drawings}
\label{sec:simple-fourone}

We begin with the construction of simple \fourone-free drawings on the torus with a quadratic number of edges. 

\begin{construction}\label{con:fourone_simple_torus}
    For infinitely many values of $n$, there is a simple $n$-vertex graph on $\Omega(n^2)$
    edges admitting a simple \fourone-free drawing on the torus.
\end{construction}

\begin{figure}[htb]
    \centering
    \begin{subfigure}{.49\textwidth}
        \centering
        \includegraphics[scale=1.6,page=1]{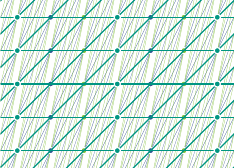}
    \end{subfigure}
    \begin{subfigure}{.49\textwidth}
        \centering
        \includegraphics[scale=1.6,page=2]{figures/no-4_1-toroidal-full.pdf}
    \end{subfigure}
    \caption{Part of the initial grid construction in \cref{con:fourone_simple_torus} for $k=2$. The figures highlight two of the three sets of edges with distinct slopes. The color of an edge matches the color of its two incident vertices, which are three columns apart.}
    \label{fig:no-4_1-torodial_full}
\end{figure}

\begin{proof}
    Let $k$ be some positive integer.
    We describe a toroidal drawing of a $2k$-regular graph on $180k$ vertices, giving the desired quadratic edge density.
    To this end, place the vertices on a toroidal grid with $20k$ rows and $9$ columns. 
    Let $v_{i,j}$ denote the vertex in row $i$ and column $j$. 
    All upcoming index operations in this notation will be modulo $20k$ in the row index and modulo $9$ in the column index.
    We assign to each column~$j$ a set $S_j$ of $k$ slopes, such that for any three consecutive columns $j,j+1,j+2$, the slope sets $S_j,S_{j+1},S_{j+2}$ are pairwise disjoint.
    Specifically, we set $S_j = \{0,\ldots,k-1\}$ for columns $j \in \{0,3,6\}$, $S_j = \{k,\ldots,2k-1\}$ for columns $j \in \{1,4,7\}$, and 
    $S_j = \{2k,\ldots,3k-1\}$ for columns $j \in \{2,5,8\}$.
    
    Now, we add for each vertex $v$ in column $j$ and each of the $k$ slopes $s \in S_j$ an edge with slope $s$ from $v$ to a vertex in column $j+3$.
    That is, for each vertex $v_{i,j}$, and for each slope $s \in S_j$, we add an edge with slope $s$ between $v_{i,j}$ and $v_{i+3s,j+3}$ that passes through $v_{i+s, j+1}$ and $v_{i+2s,j+2}$.
    See \cref{fig:no-4_1-torodial_full}.
    Note that this is not yet a valid drawing by our definition, as edges pass through vertices and there are points in which more than two edges cross.
    
    To fix this, consider any vertex $v_{i,j}$.
    There are $k$ edges incident to $v_{i,j}$ that end in column $j+3$, and $k$ edges incident to $v_{i,j}$ that start in column $j-3$.
    Since $S_j = S_{j-3} = S_{j+3}$, both these sets of $k$ edges have the same $k$ slopes.
    Additionally, there is a set $A$ of $2k$ edges that pass through $v_{i,j}$ but are not incident to $v_{i,j}$.
    The $2k$ edges in $A$ have the $2k$ slopes in $\{0,\ldots,3k-1\} - S_j$.
    The situation in a small disk around $v_{i,j}$ is depicted in \cref{fig:no-4_1-toroidal-vertex} (left). 
    We perturb the $2k$ edges of $A$ slightly within the disk around $v_{i,j}$ such that locally we obtain a valid drawing with the additional property that each of the $2k$ edges emanating from $v_{i,j}$ has its first crossing with a different of the $2k$ edges in $A$ (see \cref{fig:no-4_1-toroidal-vertex} (right)).
    This ensures that none of the cells incident to $v_{i,j}$ is a \fourone-cell.
    Repeating this procedure at every vertex $v_{i,j}$, the resulting drawing is \fourone-free.
    At each remaining point (in the space between the columns) in which more than two edges cross, we similarly locally perturb the edges to eventually obtain a valid drawing $\Gamma(k)$.

\begin{figure}[htb]
    \centering
    \includegraphics[scale=1]{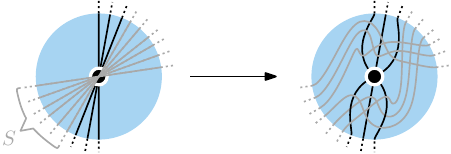}
    \caption{Illustration for \cref{con:fourone_simple_torus}. In a drawing, when a set $A$ of $2k$ edges intersect at a vertex of degree $2k$, the edges can be locally redrawn without introducing new edge intersections such that the vertex is not incident to any \fourone-cells.}
    \label{fig:no-4_1-toroidal-vertex}
\end{figure}

    Finally, we argue that $\Gamma(k)$ is a simple drawing.
    Note that any edge in $\Gamma(k)$ spans (including its endpoints) at most $9k$ rows and exactly $4$ columns.
    Hence the size of the grid ensures that two edges do not intersect multiple times by wrapping around the torus.
    Further, before the perturbations, all edges locally behave like straight-line segments, and hence any two (unperturbed) edges intersect at most once.
    Because the perturbations do not introduce any new crossings, $\Gamma(k)$ is simple. 
\end{proof}

Based on the toroidal construction, we obtain simple \fourone-free drawings in the plane with $\Omega(n^{3/2})$ edges. While this leaves a significant gap to the trivial upper bound of $\binom{n}{2}$, we believe that there are \fourone-free drawings with $\Omega(n^2)$ edges, see \cref{conj:fourone-simple}.

\begin{construction}\label{con:fourone_simple_plane}
    For infinitely many values of $n$, there is a simple $n$-vertex graph on $\Omega(n^{3/2})$
    edges admitting a simple \fourone-free drawing.
\end{construction}
\begin{proof}
    Let $k$ be some positive integer. Start with the toroidal drawing $\Gamma(k)$ described in \cref{con:fourone_simple_torus} but with an underlying $20k \times 9k$-grid instead of a $20k \times 9$-grid.
    See again \cref{fig:no-4_1-torodial_full} for a visualization.
    Now, since we cannot place $\Gamma(k)$ in the plane as is, we cut the torus between two neighboring columns (say between column $9k-1$ and $0$). The cut is positioned so that it does not pass through any edge crossings.
    The resulting ``drawing'' can be embedded in the plane but contains many halved edges that end on the boundaries of the cut we performed.
    We take care of these half-edges by placing a new vertex each at the point in the plane where the half-edge ends, and, in doing so, add degree 1-vertices.
    
    This leaves us with our final drawing $\Gamma'(k)$. Simplicity of $\Gamma'(k)$ carries over from simplicity of $\Gamma(k)$, since no new crossings are created. 
    Further, \fourone-freeness of $\Gamma'(k)$ at preexisting vertices carries over from $\Gamma(k)$, while all the new vertices have degree 1 and are therefore not incident to a \fourone-cell.
    Finally, for the edge density, we count how many degree 1-vertices are added during the construction.
    The edges that are cut in half are exactly those that are incident to one vertex in column $0, 1$, or $2$ and one vertex in column $9k-3, 9k-2$, or $9k-1$.
    There are $3 \cdot 20k \cdot k = 60k^2$ such edges. 
    For each of them, we introduce two vertices of degree 1 and increase the number of edges by one.
    On the other hand, every vertex already present in $\Gamma(k)$ has degree $2k$.
    In total, $\Gamma'(k)$ contains $20k \cdot 9k + 2 \cdot 60k^2 = \Theta(k^2)$ vertices and, $20k \cdot 9k \cdot k + 60k^2 = \Theta(k^3)$ edges, giving the claimed edge density.
\end{proof}

%%%%%%%%%%%%%%%%%%%%%%%%%%%%%%
%%                          %%
%%  NON-HOMOTOPIC DRAWINGS  %%
%%                          %%
%%%%%%%%%%%%%%%%%%%%%%%%%%%%%%

\subsection{Non-Homotopic \texorpdfstring{\fourone}{4\_1}-Free Drawings}
\label{sec:non-hom-fourone}

In the non-homotopic setting, the following construction settles the order of growth of the edge density of \fourone-free drawings by providing a quadratic lower bound. Note that in the conference version of the present paper \cite[Construction 9]{hahn_edge_densities_gd}, we incorrectly claimed that this construction yields a simple drawing.

\begin{construction}\label{con:fourone-non-hom}
    For infinitely many values of $n$, there is a simple $n$-vertex graph on $\frac{n^2}{18} - \calO(n)$ edges admitting a non-homotopic \fourone-free drawing.
\end{construction}
\begin{proof}
    For each integer $k \geq 2$, we construct a simple $k$-regular graph $G(k)$ on $n = 3 \cdot (3k-1)$ vertices together with a non-homotopic \fourone-free drawing; see \cref{fig:no-4_1-cell-component,fig:no-4_1-cell-full-graph} for an illustration.
    The graph $G(k)$ consists of three copies of a graph $G'(k)$ on $3k-1$ vertices.
    We describe $G'(k)$ in tandem with a temporary \enquote{drawing} that does not yet fulfill our definition of a drawing. 
    We start with a set $A$ of $3k-1$ equidistant points on a circle $C_{\rm out}$.
    The points in $A$ will be the vertices of $G'(k)$, which we denote $a_0,\ldots,a_{3k-2}$ in cyclic ordering around $C_{\rm out}$.
    For each $i = 0,\ldots,3k-2$ draw a straight-line edge $e_i$ between $a_i$ and $a_{i+k}$ (throughout this construction all indices are taken modulo $3k-1$).
    Let $b_i$ be the crossing point of $e_i$ and $e_{i+1}$, and $B$ the set of all such $b_i$.
    
        \begin{figure}[htb]
        \centering
        \includegraphics{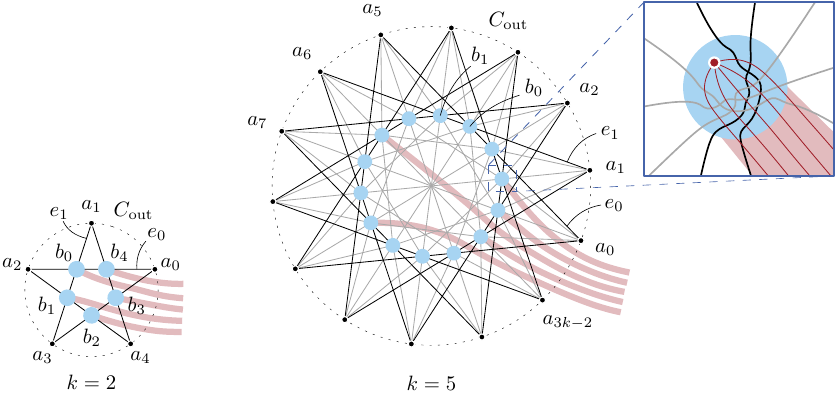}
        \caption{
            Illustration of the $k$-regular graph $G'(k)$ on $3k-1$ vertices for $k=2$ and $k=5$.
            At each of the $3k-1$ points in $B$ (blue disks), $k$ edges are pairwise crossing.
            The zoomed-in subfigure depicts the situation at a point $b \in B$ after interleaving the three copies of $G'(k)$ to obtain $G(k)$ (see also \cref{fig:no-4_1-cell-full-graph}).
            In particular, the red elements (including the corridors) belong to a different copy of $G'(k)$.
            For visual clarity, only some corridors are shown for $k=5$.
        }
        \label{fig:no-4_1-cell-component}
    \end{figure}

    Then $b_0,\ldots,b_{3k-2}$ are equidistant points (blue disks in \cref{fig:no-4_1-cell-component}) on a smaller circle $C_{\rm in}$ concentric to $C_{\rm out}$.
    Note that $e_i$ crosses $e_{i+1}$ at $b_i$ and $e_{i-1}$ at $b_{i-1}$.
    Next, we add further edges such that each $a_i$ has exactly $k$ incident edges and at each $b_i$, exactly $k$ edges pairwise cross (thus making the result not a drawing in our sense).
    To this end, we insert edges from $a_i$ to $a_{i+k+1},\ldots, a_{i+\lfloor 3k/2 \rfloor}$ where the edge $a_ia_{i+k+j}$ with $j=1,\ldots,\lfloor k/2 \rfloor$ goes through $b_{i-j-1}$ and $b_{i+2j}$; see again \cref{fig:no-4_1-cell-component} for an illustration.
    Let us emphasize (as we will need it later) that every edge passes through exactly two points in $B$. 
    This completes the construction of $G'(k)$.
    Observe that the described drawing of $G'(k)$ is simple (in fact, almost straight-line) and every vertex of $G'(k)$ lies on the same cell.

	\begin{figure}[htb]
        \centering
    \includegraphics{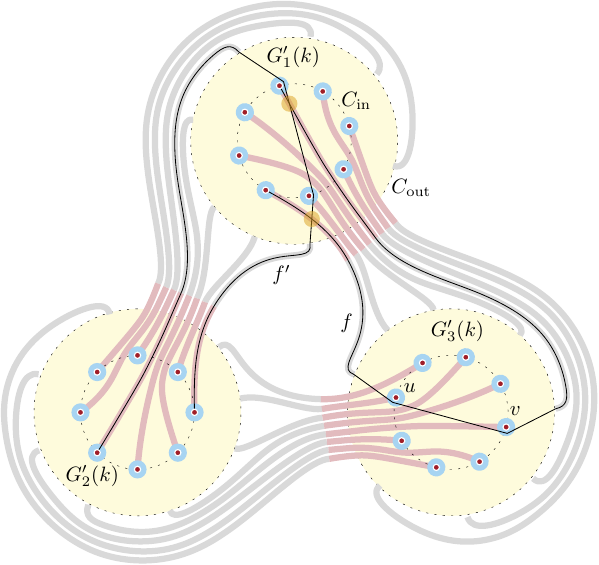}
        \caption{
            Connecting the drawings of $G'_1(k)$, $G'_2(k)$, and $G'_3(k)$ to a \fourone-free drawing of $G(k)$. The center of each $G'_i(k)$ is depicted as a yellow disk.
            The two edges $f$ and $f'$ cross twice (marked orange) but the lens they create is non-empty, since it contains the vertices $u$ and $v$.
        }
        \label{fig:no-4_1-cell-full-graph}
    \end{figure}

    We now take three copies of $G'(k)$, called $G'_1(k),G'_2(k),G'_3(k)$, initially draw them disjointly as previously described, and then modify and interleave the three \enquote{drawings} so as to obtain a valid drawing that contains no \fourone-cell.
    See \cref{fig:no-4_1-cell-full-graph} for an illustration.
    Index operations for the $G'_i(k)$ will be performed modulo~3.
    We start by placing the three copies in three disjoint disks, which we call the \emph{centers} of the respective $G'_i(k)$.
    We shall move the vertices of $G'_i(k)$ in such a way that each vertex $a$ of $G'_i(k)$ lies close to a point $b$ in the set~$B$~of~$G'_{i+1}(k)$.
    We locally perturb the $k$ edges of $G'_{i+1}(k)$ that all cross in $b$ so that each of the $k$ edges of $G'_i(k)$ emanating from $a$ has its first crossing with a different edge of $G'_{i+1}(k)$, see the zoomed-in subfigure in \cref{fig:no-4_1-cell-component}.
    Loosely speaking, the edges of $G'_{i+1}(k)$ that were crossing in a point $b$ before being perturbed ``support'' the edges of $G'_i(k)$ starting at a vertex $a$ of $G'_i(k)$.

    More precisely, we consider in the temporary drawing of $G'(k)$ for each $b \in B$ a straight and very narrow corridor starting at point $b$ and ending on $C_{\rm out}$ between $a_0$ and $a_{3k-2}$. These corridors are routed so that they intersect no existing edge twice (including those currently passing through~$b$).
    We perturb the $k$ edges close to $b$ to obtain a simple drawing in which $b$ sees a segment of each of these $k$ edges.
    This is similar to the redrawing procedure depicted in \cref{fig:no-4_1-toroidal-vertex} and ensures that the final drawing is valid.
    Now, place the simple drawings of $G'_1(k),G'_2(k),G'_3(k)$ with corridors disjointly next to each other.
    Move the vertices of $G'_1(k)$ (without introducing new crossings between edges of $G'_1(k)$) through the prepared corridors of $G'_2(k)$ (each vertex through one corridor) into the center of $G'_2(k)$.
    Place each vertex $a$ of $G'_1(k)$ at the end of the corridor of $G'_2(k)$ (at the point $b \in B$) so that each edge (of $G'_1(k)$) at $a$ crosses through a different edge (of $G'_2(k)$) near $b$.
    This ensures that $a$ is not part of any \fourone-cell.
    By similarly threading the vertices of $G'_2(k)$ through the corridors of $G'_3(k)$ and the vertices of $G'_3(k)$ through the corridors of $G'_1(k)$, \fourone-freeness is ensured.
    Note that every edge of $G'_i(k)$ now passes exactly two vertices of $G'_{i-1}(k)$ that are now placed within the center of $G'_i(k)$.
    This finishes the construction of the drawing of $G(k)$.

    We are left with arguing non-homotopicity.
    First note that, while each component is itself drawn simply and no corridor crosses any edge twice, the resulting drawing is not simple.
    The reason for this is that every edge passes through two corridors (near its two incident vertices), which may cross a common set of edges (see again \cref{fig:no-4_1-cell-full-graph}).
    More precisely, all violations to simplicity occur because an edge $f$ of $G_i'(k)$ that starts and ends in the center of $G'_{i+1}(k)$ is crossed twice by an edge $f'$ of $G'_{i+1}(k)$ within the center of $G'_{i+1}(k)$.
    To argue that the drawing is non-homotopic, we show that any such lens~$\ell$ bounded by $f$ and $f'$ is non-empty.
    Note that, by construction, $f$ passes near two vertices $u$ and $v$ of $G_{i-1}'(k)$ within the center of $G'_i(k)$.
    We will show that $u$ and $v$ lie inside the lens $\ell$.
    For contradiction's sake, assume that one of the two vertices, say $u$, was not contained in the interior of $\ell$.
    Again, by construction, $f$ ``supports`` an edge $g$ incident to $u$. 
    That is, $g$ has its first crossing with $f$, where by assumption it enters $\ell$.
    Because of the way the copies are interleaved, $g$ has to follow a corridor to enter the center of $G_{i-1}'(k)$ after crossing $f$.
    However, none of the corridors that connect the center of $G_{i}'(k)$ to the center of $G_{i-1}'(k)$ are contained in the interior of $\ell$, except for their initial segments inside the center of $G_{i}'(k)$.
    Hence, in order to follow a corridor into the center of $G_{i-1}(k)$, $g$ needs to cross $f$ again within the center of $G_{i}'(k)$. 
    But then the corridor through which $g$ travels crosses $f$ twice, a contradiction to the fact that the corridors are drawn simply with respect to the initial drawings of $G'(k)$.
    This contradiction shows that $\ell$ contains $u$ (and also $v$) in its interior.
    The same argument holds for every lens in the drawing, and thus, our drawing of $G(k)$ is non-homotopic.
\end{proof}

%%%%%%%%%%%%%%%%%%%%%%%%%%%%%%
%%                          %%
%%    DRAWING ALL GRAPHS    %%
%%                          %%
%%%%%%%%%%%%%%%%%%%%%%%%%%%%%%

\section{Drawing All Graphs without Some Cell Type}
\label{sec:all-graphs}

In \cref{sec:complete-graphs}, we have seen that most cell types can be avoided in simple drawings of arbitrarily large complete graphs. 
However, this does not mean that \emph{all} simple graphs can be drawn when one of these cell types is forbidden.
The issue is that the graph property of admitting a \celltype-free drawing is not closed under taking subgraphs.
For example, $K_5$ admits a simple drawing without \sixthree-cells, while $K_3$ does not.
Thus, the structure of all graphs with (simple, non-homotopic) \celltype-free drawings could be quite complex.

In this section, we show that for any type of cell \celltype, whose boundary is not incident to a crossing, this is not the case: For any such \celltype, we show that all but a finite number of connected simple graphs admit a simple/non-homotopic \celltype-free drawing.
In fact, the next proposition shows that almost all connected simple graphs admit a simple drawing in which every cell is incident to a crossing. The main idea is to consider crossing-maximal convex drawings.

\begin{theorem}
\label{thm:no-cycle-cells}
    Let $G$ be a connected simple graph with $G \notin \{K_3, K_4\} \cup \{K_{1,n} \mid n \geq 0\}$.
    Then $G$ admits a simple drawing in which every cell is incident to at least one crossing.
\end{theorem}
\begin{proof}
    For complete graphs on $n \geq 5$ vertices, the crossing-maximal drawings in \cite{harborthEdgesWithoutCrossings1974} fulfill the requirements.
    Thus, from now on we may assume that $G = (V,E)$ is a connected simple graph that is neither a star nor complete and contains at least $4$ vertices.
    Consider a straight-line drawing $\Gamma$ of~$G$ in which all vertices are placed on a common circle $C$ such that $\Gamma$ has the maximum number of crossings among all such drawings.
    Note that two edges in $\Gamma$ cross if and only if the order of their vertices along $C$ alternates between the two edges.

    We first prove that all interior cells in $\Gamma$ are incident to a crossing.
    Suppose for the sake of contradiction that $\Gamma$ contains an interior cell $c$ bounded by a plane $k$-cycle $H \subseteq G$, where $k \geq 3$.
   
    If $G = H$, then $k \geq 4$. Switching the positions of any non-adjacent pair of vertices on $C$ increases the number of crossings, thus contradicting our assumption that $\Gamma$ was crossing-maximal. 
    
    Hence we may assume that $G \neq H$. 
    Thus, there is an edge $uv \in E(H)$ not on the convex hull of $V$.
    Let $S$ be the set of vertices on the part of the circle $C$ bounded by (but not including) $u$ and $v$ that contains all other vertices of $H$.
    As $G \neq H$ and $G$ is connected, we may assume that $u$ has a neighbor $w$ with $w \notin S \cup \{u,v\}$.
    Let $R$ denote the part of the circle $C$ bounded by $u$ and $w$ that does not include $v$.

    Let $\Gamma'$ be the drawing that is obtained by moving all vertices in $S$ onto $R$ such that their clockwise order along $C$ is reversed relative to $\Gamma$. See \cref{fig:no-cycle-cells} (top left) for an illustration.
    We argue that there are more crossings in $\Gamma'$ than in $\Gamma$.
    Note that, since $H$ is an uncrossed cycle in $\Gamma$, all crossings in $\Gamma$ are either spanned by four vertices in $V \setminus S$ or by four vertices in $S \cup \{u,v\}$.
    All these crossings are also present in $\Gamma'$, since the subdrawings of $\Gamma$ and $\Gamma'$ on $V \setminus S$ and $S \cup \{u,v\}$, respectively, have the same vertex orders along $C$ (up to global reflection).
    Additionally, in $\Gamma'$, the edge $v v' \in E(H)$ with $v' \neq u$ crosses $uw$.
    This contradicts crossing-maximality of $\Gamma$, so our assumption was faulty. Thus, all interior cells in~$\Gamma$ are incident to a crossing.

    \begin{figure}[htb]
        \centering
        \begin{subfigure}{.49\textwidth}
            \centering
            \includegraphics[page=1]{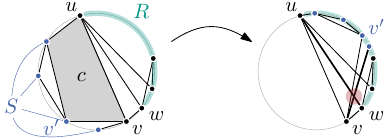}
        \end{subfigure}~
        \begin{subfigure}{.49\textwidth}
            \centering
            \includegraphics[page=2]{no-cycle-cells-lb.pdf}
        \end{subfigure}
        \\[3ex]
        \begin{subfigure}{.49\textwidth}
            \centering
            \includegraphics[page=3]{no-cycle-cells-lb.pdf}
        \end{subfigure}
        \begin{subfigure}{.49\textwidth}
            \centering
            \includegraphics[page=4]{no-cycle-cells-lb.pdf}
        \end{subfigure}
        \caption{
            Illustrations for \cref{thm:no-cycle-cells}. Top left: Arguing that a crossing-maximal convex drawing has no interior uncrossed cells. Top right: Ensuring the outer cell is incident to a crossing in Case 1. Bottom: Arguing that the outer cell is incident to a crossing in Case 2.
        }
        \label{fig:no-cycle-cells}
    \end{figure}

    If the outer cell in $\Gamma$ is also incident to a crossing, we are done.
    Otherwise, we distinguish two cases. In the first case we will use that $G$ is not complete while in the second case we will need that $G$ is not a star graph.
    
    \textbf{Case 1:} Every edge of the convex hull of $V$ is an edge of $G$.
    Let $v_0, \ldots, v_{n-1}$ be the vertices on $C$ in cyclic order. 
    All subsequent index operations are implicitly performed modulo $n$. 
    Let $v_i$ and $v_j$ be two non-adjacent vertices such that $j-i$ is minimal among all non-adjacent vertex pairs.
    Note that such a pair exists because $G$ is not complete and that the current case implies $j-i \geq 2$. Hence, placing $v_j$ in the space between $v_i$ and $v_{i+1}$ on $C$ gives a new drawing $\Gamma'$; see \cref{fig:no-cycle-cells} (top right) for a depiction.
    Since $v_i v_j \notin E$, the outer cell in $\Gamma'$ is incident to a segment of the edge $v_i v_{i+1}$.
    Since further $v_i v_{i+1}$ is crossed by $v_j v_{j+1}$ in $\Gamma'$, this ensures that the outer cell is incident to a crossing. 
    To finish Case~1, we are left with verifying that the move of $v_j$ between $\Gamma$ to $\Gamma'$ did not introduce any uncrossed interior cells.
    For this, suppose there is an interior cell in $\Gamma'$ bounded by a cycle $H \subseteq G$.
    The moved vertex $v_j$ cannot be in $H$, as all its incident edges except $v_j v_{i+1}$ are crossed by $v_i v_{i+1}$.
    Similarly, $v_{i+1}, \ldots, v_{j-2}$ are not part of $H$ since $v_{i+1}, \ldots, v_j$ form a convex drawing of a complete graph in $\Gamma'$ (otherwise $v_i v_j$ was not a shortest missing edge).
    Further, all edges between $v_{j-1}$ and vertices $v_{j+2}, \ldots, v_i$ except the edge $v_{j-1} v_{j+1}$ (if it exists) are crossed by the edge $v_j v_{j+1}$ in $\Gamma'$, implying that, again, $v_{j-1}$ is not contained in $H$.
    Therefore, $H$ contains only vertices in $v_{j+1}, v_{j+2},\ldots, v_i$.
    However, since $v_j$ does not change its position with respect to these vertices, $H$ also bounds a cell in $\Gamma$. 
    This contradicts our assumption that all inner cells of $\Gamma$ have an incident crossing.
    Hence, $\Gamma'$ is indeed a drawing of $G$ with the desired properties.
    
    \textbf{Case 2:} In $\Gamma$, two neighbors $u$ and $v$ on $C$ are not adjacent in $G$.
    Without loss of generality, let $u$ precede $v$ in clockwise order.
    Further, let $u'$ and $v'$ be the first neighbor along the boundary of the outer cell of $\Gamma$ starting clockwise from $u$ and counterclockwise from $v$, respectively . Since there are no crossings on the outer face of $\Gamma$, both $u'$ and $v'$ are cut-vertices in $G$, disconnecting $u$ and $v$ when removed.
    If $u' \neq v'$, then switching the positions of $u$ and $v$ introduces a new crossing between $u u'$ and $v v'$, while retaining all other crossings, see \cref{fig:no-cycle-cells} (bottom left) --- a contradiction to the crossing-maximality of $\Gamma$.

    So, we may assume that $u' = v'$.
    Moreover, by the previous reasoning, we may assume that any two non-adjacent vertices of $G$ that are neighbors on $C$ in $\Gamma$ are adjacent to a common cut-vertex.
    If both $u$ and $v$ have degree 1, then walking along $C$ eventually gives two non-adjacent neighboring vertices on $C$, both adjacent to $u'$ and one of which of degree at least 2. 
    This is guaranteed by the previous assumption and since $G$ is not a star.
    Without loss of generality, let $\deg(v) > 1$. 
    Then, switching the position of $u$ and $v$ on $C$ increases the number of crossings, since the edge $w v$ to any vertex $w \neq u'$ adjacent to $v$ in $G$ now crosses $u u'$. See \cref{fig:no-cycle-cells} (bottom right). Once more, we reached a contradiction to crossing-maximality of $\Gamma$.
    Therefore, the outer cell of $\Gamma$ is also incident to a crossing and the claim is proven.
\end{proof}

With \cref{thm:no-cycle-cells} at hand, we can fully characterize the simple connected graphs that admit a simple \celltype-free drawing for every (non-degenerate) cell type \celltype whose boundary is not incident to any crossing.

\begin{corollary}\label{cor:no-cycle-cells}
    For an integer $m \geq 2$, let $\celltype_m$ denote the cell type whose boundary is connected and bounded by $m$ vertices and $m$ uncrossed edge segments.
    Let $G$ be a connected simple graph.
    Then $G$ admits a simple $\celltype_m$-free drawing, unless
    \begin{itemize}
        \item $m = 3$ and $G = K_3$, or
        \item $m$ is even and $G=K_{1,m/2}$.
    \end{itemize}
\end{corollary}
\begin{proof}
    For all connected simple graphs $G$ with $G \notin \{K_3, K_4\} \cup \{K_{1,k} \mid k \geq 0\}$, the statement is an immediate consequence of \cref{thm:no-cycle-cells}.
    For $K_4$, a plane drawing contains only $\celltype_3$-cells and a convex straight-line drawing consists of a $\celltype_4$-cell and four \fivetwo-cells, so any $\celltype_m$ can be avoided by choosing one of the two drawings.
    The unique drawing of $K_{1,0}=K_1$ has no cell of type $\celltype_m$ for any $m\ge 2$.
    Finally, for the complete graph $K_3$ and the star graphs $K_{1,k}$, the unique (up to homeomorphism of the plane) simple drawings contain only the cell type $\celltype_3$ and $\celltype_{2k}$, respectively, 
    which both explains the exceptions and shows the statement also for these graphs. 
\end{proof}

We remark that \cref{thm:no-cycle-cells} can be extended to disconnected graphs.
In this case, if a graph $G$ contains two components with at least one edge each, there is a simple drawing of $G$ in which every cell is incident to a crossing.
If otherwise, $G$ consists of a graph $H$ and some isolated vertices, $G$ admits such a drawing if and only if $H$ does.

\cref{cor:no-cycle-cells} can be generalized to disconnected graphs in a similar way.
For connected cell types without crossings, the only disconnected exception graph is $K_3 + K_1$, which does not admit a $\celltype_3$-free drawing. However, one might also want to consider cell types with disconnected boundary and no incident crossing when including disconnected graphs. 
Then, the only exception graphs are $K_3$, $K_3 + K_1$ and $S + nK_1$ for any star graph $S$ and $n \geq 0$, with respect to the cells that appear in their unique drawings.

%%%%%%%%%%%%%%%%%%%%%%%%%%%%%%
%%                          %%
%%    CONCLUDING REMARKS    %%
%%                          %%
%%%%%%%%%%%%%%%%%%%%%%%%%%%%%%
\section{Concluding Remarks}
\label{sec:conclusion}

With this paper, we initiated the study of drawings of graphs with a forbidden type of cell.
For arbitrary, non-homotopic, and simple drawings, and all cell types $\celltype$ but \fourzero, we determined whether graphs that admit a \celltype-free drawing have at most linear or superlinear edge density.
In many cases we obtain near-tight bounds.
Being unable to control the \fourzero-cell with our methods and constructions, we wonder in which regime the number of edges in \fourzero-free drawings lies.

\begin{problem}
    \label{prob:fourzero_density}
    Is the edge density of simple graphs that admit a (simple, non-homotopic) \fourzero-free drawing quadratic in $n$?
\end{problem}

For simple \fourone-free drawings, a gap remains between the $\Omega(n^{3/2})$ lower bound (\cref{con:fourone_simple_plane}) and the trivial $\mathcal{O}(n^2)$ upper bound.
In light of our quadratic lower bounds in the relaxed settings of toroidal simple drawings (\cref{con:fourone_simple_torus}) and non-homotopic drawings (\cref{con:fourone-non-hom}), we believe the true edge density for simple \fourone-free drawings in the plane should be quadratic.

\begin{conjecture}
    \label{conj:fourone-simple}
    The edge density of $n$-vertex graphs that admit a simple \fourone-free drawing is $\Omega(n^2)$.  
\end{conjecture}

The vast number of questions asked about other beyond planar graph classes -- for example, about complexity of recognition and stretchability -- can be carried over to our setting.
However, our forbidden patterns rely on the topology of concrete drawings rather than only abstract combinatorics of the graph or the drawing (such as crossing edge pairs).
Hence, the resulting graph classes are not always closed under taking subgraphs.
In light of this, the results that almost all cell types can be avoided in drawings of $K_n$ from \cref{sec:complete-graphs} raise the question of which cell types can be avoided in drawings of any graph.
For cell types whose boundary does not contain any crossing, we gave an answer in \cref{sec:all-graphs} but a full characterization remains elusive.

\begin{problem}
    For which cell types \celltype does every (connected) graph admit a (simple, non-homotopic) \celltype-free drawing?
\end{problem}

Apart from our contributions about forbidden cells, we also improved the lower bound on the edge density of simple graphs  that admit a quasiplanar drawings from $7n - \mathcal{O}(1)$ to $7.5n - \mathcal{O}(1)$.
However, a linear gap remains to the upper bound of $8n - 20$.

\begin{problem}
    Are there quasiplanar drawings of $n$-vertex simple graphs with $8n- \mathcal{O}(1)$ edges?
\end{problem}

%% --------------------------------------------------------------------
%       Acknowledgements
%% --------------------------------------------------------------------

 \section*{Acknowledgements}
We thank the various anonymous referees, who, through their valuable comments, helped to significantly improve the presentation of this paper.

\bibliography{02_bibliography.bib}
\end{document}